\documentclass[11pt]{amsart}

\usepackage{latexsym}
\usepackage{amssymb}
\usepackage{amsmath}

\newtheorem{theorem}{Theorem}[section]
\newtheorem{lemma}[theorem]{Lemma}
\newtheorem{proposition}[theorem]{Proposition}
\newtheorem{corollary}[theorem]{Corollary}
\newtheorem{definition}[theorem]{Definition}

\newtheorem{remark}[theorem]{Remark}

\newcommand\nph{\varphi}

\newcommand\Bim{\mathop{\rm Bim}}
\newcommand\Sat{\mathop{\rm Sat}}

\newcommand\vn{\mathop{\rm VN}}

\newcommand\cb{\mathop{\rm cb}}

\newcommand{\cl}[1]{\mathcal{#1}}
\newcommand{\bb}[1]{\mathbb{#1}}

\newcommand{\sca}[1]{\left(#1\right)}
\newcommand{\nor}[1]{\left\Vert #1\right\Vert}

\begin{document}

\title{Transference and preservation of uniqueness}

\author{I. G. Todorov}
\address{Pure Mathematics Research Centre, Queen's University Belfast, Belfast BT7 1NN, United Kingdom}
\email {i.todorov@qub.ac.uk}

\author{L. Turowska}
\address{Department of Mathematical Sciences,
Chalmers University of Technology and  the University of Gothenburg,
Gothenburg SE-412 96, Sweden}
\email{turowska@chalmers.se}

\date{21 August 2016}

\begin{abstract}
Motivated by the notion of a set of uniqueness in a locally compact group $G$, we introduce and
study ideals of uniqueness in the Fourier algebra $A(G)$ of $G$,
and their accompanying operator version, masa-bimodules of uniqueness.
We establish a transference between the two notions, and use this result
to show that the property of being an ideal of uniqueness is
preserved under natural operations.
\end{abstract}

\maketitle

\section{Introduction}\label{s_intro}

The notion of a set of uniqueness in the group of the circle arises
in connection with the problem of uniqueness of
trigonometric expansions and goes back to
Cantor. It has been studied extensively in the context of abelian locally compact groups (see
\cite{gmcgehee}).  It was extended to arbitrary locally compact (not 
necessarily commutative) groups by M. Bo\.{z}ejko \cite{bozejko_pams}
and was shown in \cite{stt} to play a decisive role in questions about closability of multipliers on 
group  C*-algebras.

Motivated by W. B. Arveson's pivotal paper \cite{a},
a programme of establishing precise links between
harmonic analytic and operator algebraic notions has been pursued
since the 1970's,
allowing the transference of fundamental concepts from Harmonic Analysis
to the setting of operator algebras
(see \cite{t_serdica} for a survey).
In addition, operator theoretic methods have been successfully employed to
obtain results belonging to the area of Harmonic Analysis {\it per se} (see {\it e.g.} \cite{et2}).
These ideas, along with questions about closability
of operator transformers,
led to the study of
sets of operator uniqueness in \cite{stt}, where
it was shown that a closed subset $E$ of a second countable
locally compact group $G$ is a set of uniqueness if and only if the subset
$E^* = \{(s,t) : ts^{-1}\in E\}$ of $G\times G$ is a set of
operator uniqueness. 

Along with sets of uniqueness, the accompanying (and weaker) notion of
a $U_1$-set has been investigated in the literature (see {\it e.g.} \cite{gmcgehee} and \cite{ulger}).  
A corresponding operator theoretic version, that of an operator $U_1$-set, 
was defined in \cite{stt_IEOT} and subsequently studied in \cite{stt}, 
where it was shown that a closed set $E\subseteq G$ 
is a $U_1$-set if and only if 
$E^*$ is an operator $U_1$-set.

It is known that, given a closed subset $E$ of a locally compact group $G$,
there exist two extremal closed ideals of the Fourier algebra $A(G)$ of $G$ whose null set coincides with $E$.
The notions of a set of uniqueness and of a $U_1$-set are defined through these ideals. 
In this paper, we unify the two concepts by defining and studying the notion of an 
ideal of uniqueness. 
We place in a general setting
a number of concepts from Abstract Harmonic Analysis and Operator Algebra Theory, 
obtaining as special cases some of the main results of \cite{stt} and \cite{stt_IEOT}.
The approach we take allows us to consider sets of uniqueness
and $U_1$-sets as special cases of the same concept,
and consequently to treat their properties in a consolidated manner.

In more detail, the content of the paper is as follows.
In Section \ref{s_prel}, we set notation and provide some necessary background,
including Eymard's approach that allows one to view $A(G)$ as the predual
of the von Neumann algebra $\vn(G)$ of $G$,
much of which lies at the heart of our development.
In Section \ref{s_cr}, we define the notions of ideals of multiplicity and uniqueness,
and their operator versions, namely, the notions of
masa-bimodules of multiplicity and uniqueness.
We establish a norm closure version of the main result from \cite{akt},
showing that the compact operators in the weak* closed $L^{\infty}(G)$-bimodule
generated by the annihilator $J^{\perp}$ of
a closed ideal $J$ of $A(G)$ can be
approximated in norm by saturations of the elements of the intersection
$J^{\perp}\cap C_r^*(G)$, where $C_r^*(G)$ is the reduced C*-algebra of $G$.
As a consequence of this result, we show that $J$ is an ideal of multiplicity
if and only if the weak* closed $L^{\infty}(G)$-bimodule generated by $J^{\perp}$ is a masa-bimodule
of multiplicity. A special case of this result is the theorem on transference of uniqueness for closed subsets of $G$
established in \cite[Theorem 4.9]{stt}.

In \cite[Corollary 4.14]{stt}, it was shown that
the property of being a set of uniqueness is preserved under taking (finite) unions.
In Section \ref{s_pres}, we generalise this result by
showing that the intersection of two ideals of uniqueness is an ideal of uniqueness.
The result is obtained as a consequence of
an intersection identity about masa-bimodules (Theorem \ref{th_sumsink})
that may be interesting in its own right; indeed, subtle intersection formulas 
for masa-bimodules were at the heart of operator algebraic applications to spectral synthesis
found in \cite{et2}.
We similarly extend the fact that the product of two sets of multiplicity is a set of
multiplicity \cite[Corollary 5.12]{stt} to a corresponding preservation property for ideals (Theorem \ref{th_tp}),
and establish preservation under inverse images (Theorem \ref{th_invima})
which should be compared to \cite[Corollary 5.7]{stt}.

\section{Preliminaries}\label{s_prel}

In this section, we introduce notation and include some background results that will be used in the
sequel.
For a Banach space $\cl X$, we denote by $\cl X^*$ its Banach space dual,
and for a subspace $\cl Y\subseteq \cl X$ (resp. $\cl Z\subseteq \cl X^*$),
we let
$$\cl Y^{\perp} = \{\omega\in \cl X^* : \omega(y) = 0, \mbox{ for all } y\in \cl Y\}$$
(resp.
$$\cl Z_{\perp} = \{x\in \cl X : \omega(x) = 0, \mbox{ for all } \omega\in \cl Z\})$$
be its annihilator (resp. preannihilator).
If $H$ is a Hilbert space, we denote by
$\cl B(H)$ the C*-algebra of all bounded linear operators on $H$.
Throughout the paper, $G$ will denote a second countable locally compact group.
We denote by $L^p(G)$, $p = 1,2,\infty$, the corresponding
Lebesgue spaces with respect to a fixed left Haar measure on $G$.
The left regular representation $\lambda : G\to \cl B(L^2(G))$ of $G$, given by
$\lambda_t f(s) = f(t^{-1}s)$, $f\in L^2(G)$, $s,t\in G$,
lifts to a *-representation of $L^1(G)$ on $L^2(G)$ {\it via} the formulas
$$\lambda(f) g(t) = (f\ast g)(t) = \int f(s) g(s^{-1}t)ds, \ \ f\in L^1(G),  g\in L^2(G), t\in G.$$
We denote by $C_r^*(G)$ the \emph{reduced C*-algebra} of $G$, that is,
the closure of $\lambda(L^1(G))$ in the operator norm, and let
$\vn(G)$ be the \emph{von Neumann algebra} of $G$, that is, the closure of $C_r^*(G)$ in the
weak* topology.
The \emph{Fourier algebra}
$A(G)$ of $G$ \cite{eymard} consists of all functions of the form $s\rightarrow (\lambda_s\xi,\eta)$,
where $\xi,\eta\in L^2(G)$.
It is known \cite{eymard} that $A(G)$ is a semisimple, regular, commutative
Banach algebra with spectrum $G$
that can be identified with the predual of $\vn(G)$ via the pairing
$\langle u,T\rangle = (T\xi,\eta)$, where $\xi,\eta\in L^2(G)$ are such that
$u(s) = (\lambda_s\xi,\eta)$, $s\in G$.
We denote by $\|u\|_A$ the norm of an element $u$ of $A(G)$.

Let
$$MA(G) = \{v : G\to\bb{C} \ : \ vu\in A(G), \mbox{ for all } u\in A(G)\}$$
be the multiplier algebra of $A(G)$.
For each $v\in MA(G)$, the map $u\to vu$ on $A(G)$ is bounded;
as usual, let $M^{\cb}A(G)$ be the subalgebra of $MA(G)$ consisting
of the elements $v$ for which the map $u \to vu$ on $A(G)$ is completely bounded
\cite{de-canniere-haagerup} (here, $A(G)$ is given the
operator space structure arising from the identification $A(G)^* \equiv \vn(G)$).
We equip $M^{\cb}A(G)$ with the corresponding completely bounded norm.

Let $(X,\mu)$ be a standard $\sigma$-finite measure space.
For a subset $\alpha\subseteq X$, we denote by $\chi_{\alpha}$ the
characteristic function of $\alpha$.
For a function $a\in L^{\infty}(X)$, we let $M_a$ be the (bounded)
operator on $L^2(X)$ given by $M_a(\xi) = a\xi$. We set
$$\cl D = \{M_a : a\in L^{\infty}(X)\};$$
thus, $\cl D$ is a maximal abelian selfadjoint algebra (masa, for short),
acting on $L^2(X)$.
Let $T(X) = L^2(X)\hat{\otimes} L^2(X)$, where $\hat{\otimes}$ denotes the
Banach space projective tensor product.
Every element $h\in T(X)$ is an absolutely convergent series
$h=\sum_{i=1}^{\infty} f_i\otimes g_i$ for some square summable sequences
$(f_i)_{i\in \bb{N}}, (g_i)_{i\in \bb{N}} \subseteq  L^2(X)$, and
may be considered either as a function $h : X\times X\to\bb C$ given by
\[
 h(s,t)=\sum_{i=1}^{\infty} f_i(s)g_i(t),
\]
or as an element of the predual of the space $\cl B(L^2(X))$
{\it via} the pairing
\[
 \langle T, h\rangle := \sum_{i=1}^{\infty} (Tf_i,\bar g_i).
\]
We denote by $\nor{h}_T$ the norm of $h\in T(X)$.
Note that $T(X)$ can itself be identified with
the dual of the space $\cl K$ of all compact operators on $L^2(X)$.

The space $\frak{S}(X)$ of all \emph{Schur multipliers}
on $X\times X$ consists, by definition, of all
measurable essentially bounded functions $\nph : X\times X\to \bb{C}$
such that $\nph h$ is equivalent (with respect to product measure)
to a function from $T(X)$, for every $h\in T(X)$.
If $\nph\in \frak{S}(X)$ then the map $h\to \nph h$ on $T(X)$ is bounded, and
its dual on $\cl B(L^2(G))$ will be denoted by $S_{\nph}$. We let
$\|\nph\|_{\frak{S}} = \|S_{\nph}\|$.
Note that $S_{\nph}$ leaves $\cl K$ invariant and hence, if $\nph\in \frak{S}(X)$
then the map $h\to \nph h$ on $T(X)$ is weak* continuous.
It is easy to see that, if $\phi, \psi\in L^{\infty}(X)$
then the function $\phi\otimes \psi$ (given by $(\phi\otimes \psi)(s,t) = \phi(s)\psi(t)$)
belongs to $\frak{S}(X)$; thus,
$T(X)$ has a natural  $L^{\infty}(X)$-bimodule structure given by letting
$\phi\cdot h \cdot \psi = (\phi\otimes \psi)h$, $h\in T(X)$, $\phi, \psi\in L^{\infty}(X)$.
In a similar fashion, $\cl B(L^2(X))$ has the structure of a $\cl D$-bimodule, arising from operator multiplication.
A subspace $\cl U\subseteq \cl B(L^2(X))$ will be called
a \emph{$\cl D$-bimodule}, or a \emph{masa-bimodule}, if
$ATB\in \cl U$ whenever $T\in \cl U$ and $A,B\in \cl D$.

We will mostly be interested in the case where $X = G$ and $\mu$
is the left Haar measure. In this case,
the predual $P$ of the inclusion of $\vn(G)$ into $\cl B(L^2(G))$ is a contraction
from $T(G)$ onto $A(G)$ given by 
\begin{equation}\label{eq_p}
P(h)(t) = \int_G h(t^{-1} s,s)ds,  \ \ \ t\in G
\end{equation}
(see \cite{stt}).
Hence, if $T\in VN(G)$  and $h\in T(G)$, then
\begin{equation}\label{pair}
\langle T,P(h)\rangle=\langle T,h\rangle,
\end{equation}
where the first pairing is between $\vn(G)$ and $A(G)$, while the second one is that between $\cl B(L^2(G))$ and $T(G)$.

If $u\in L^{\infty}(G)$, let $N(u) : G\times G\to \bb{C}$ be given by
$N(u)(s,t) = u(ts^{-1})$.
The following fact \cite{bf} (see also \cite{j}, \cite{spronk}) is at the base of
subsequent transference results.

\begin{theorem}\label{spro}
The map $u\rightarrow N(u)$ is an isometry from $M^{\cb}A(G)$ into $ \frak{S}(G)$.
\end{theorem}

\section{Transference}\label{s_cr}

In this section, we introduce ideals and bimodules of multiplicity
and establish a transference result that permits passing from the former to the latter notion.

\begin{definition}\label{d_mi}
A closed ideal $J$ of $A(G)$ will be called an \emph{ideal of multiplicity} if
$J^{\perp}\cap C^*_r(G) \neq \{0\}$. The ideal $J$ will be called an
\emph{ideal of uniqueness} if it is not an ideal of multiplicity.
\end{definition}

\noindent
{\bf Remark. } Let $J_1$ and $J_2$ be closed ideals of $A(G)$ such that $J_1\subseteq J_2$.
It is clear that if $J_1$ is an ideal of uniqueness then so is $J_2$.

\medskip

Let $E\subseteq G$ be a closed set.
It is well-known that the ideals
$$I(E)=\{u\in A(G) : u(s)=0, s\in E\}$$
and
$$J(E)=\overline{\{u\in A(G) : u \text{ has compact support disjoint from } E\}}^{\|\cdot\|_{A}}$$
are extremal in that
if $J\subseteq A(G)$ is a closed ideal whose null set is equal to $E$
(in the sense that $E = \{s\in G : u(s) = 0, \mbox{ for all } u\in J\}$),
then $J(E)\subseteq J\subseteq I(E)$.

The property of $E$ being
a \emph{set of uniqueness} was introduced in \cite{bozejko_pams}.
In \cite{stt}, following the earlier literature on classical groups,
sets that are not of uniqueness we called \emph{$M$-sets},
and the accompanying notion of \emph{$M_1$-sets} was introduced.
It is clear from Definition \ref{d_mi}, and the definitions made in \cite{bozejko_pams} and \cite{stt},
that $E$ is an $M$-set (resp. an $M_1$-set) if and only if $J(E)$ (resp. $I(E)$)
is an ideal of multiplicity.

\medskip


We next introduce the operator version of the notion of an ideal of multiplicity.

\begin{definition}\label{d_mb}
Let $(X,\mu)$ be a standard $\sigma$-finite measure space.
A weak* closed masa-bimodule $\cl U\subseteq \cl B(L^2(X))$ will be called a
\emph{bimodule of multiplicity} if $\cl U\cap \cl K\neq \{0\}$.
The space $\cl U$ will be called a
\emph{bimodule of uniqueness} if it is not a bimodule of multiplicity.
\end{definition}

\noindent
{\bf Remark } Let $\cl U_1$ and $\cl U_2$ be weak* closed masa-bimodules such that $\cl U_1\subseteq \cl U_2$.
It is clear that if $\cl U_2$ is a bimodule of uniqueness then so is $\cl U_1$.

\medskip

Recall from \cite{a} and \cite{eks} that
a subset $M\subseteq X\times X$ is called \emph{marginally null}
if there exists a null set $N\subseteq X$ such that $M\subseteq (N\times X)\cup (X\times N)$.
Two sets $\kappa,\kappa'\subseteq X\times X$ are called \emph{marginally equivalent} if
their symmetric difference $\kappa\Delta\kappa'$ is marginally null.
A subset $\kappa\subseteq X\times X$ is called \emph{$\omega$-open}
if it is marginally equivalent to a subset of $X\times X$ of the form
$\cup_{i=1}^{\infty}\alpha_i\times\beta_i$, where $\alpha_i,\beta_i\subseteq X$
are measurable. The set $\kappa$ is called \emph{$\omega$-closed} if its complement
is $\omega$-open.

Let $\kappa\subseteq X\times X$ be an $\omega$-closed set.
Analogously to the case of subsets of a locally compact group $G$, two
weak* closed masa-bimodules were introduced in \cite{a} and studied in the literature (see \cite{eks} and \cite{st1}):
$$\frak{M}_{\max}(\kappa) =
\{h\in T(X) : h \mbox{ vanishes on an }\omega\mbox{-open nbhd of } \kappa\}^{\perp}$$
and
$$\frak{M}_{\min}(\kappa) = \{h\in T(X) : h \mbox{ vanishes on } \kappa\}^{\perp}.$$
To every masa-bimodule $\cl U\subseteq \cl B(L^2(G))$, an
$\omega$-closed set $\kappa$ called its \emph{support} was associated in \cite{eks},
and it was shown that if $\cl U$ is a weak* closed masa-bimodule with support $\kappa$
then $\frak{M}_{\min}(\kappa)\subseteq \cl U\subseteq \frak{M}_{\max}(\kappa)$.

The property of $\kappa$ being an \emph{operator $M$-set}
(resp. an \emph{operator $M_1$-set}) was intoduced in \cite{stt_IEOT}.
We have that $\kappa$ is an operator $M$-set (resp. an operator $M_1$-set) precisely when
$\frak{M}_{\max}(\kappa)$ (resp. $\frak{M}_{\min}(\kappa)$) is a
bimodule of multiplicity.

\medskip

Since compact subsets of $G$ have finite Haar measure,
if $L\subseteq G$ is compact then the function $\chi_{L\times L}$ belongs to $T(G)$ and hence, by Theorem \ref{spro},
$N(u)\chi_{L\times L} \in T(G)$ for every $u\in M^{\cb}A(G)$.
Given a subspace $J\subseteq A(G)$, we let $\Sat(J)$ be the
closed $L^\infty(G)$-bimodule of $T(G)$ generated by  the set
\[
\{N(u)\chi_{L\times L} : u \in J,  \ L  \text{ compact, } \ L\subseteq G \}.
\]
It was shown in \cite{akt} that
$\Sat(J)=\overline{[N(J)T(G)]}^{\|\cdot\|_T}.$

On the other hand, given a subspace $\cl X$ of $\vn(G)$, we let $\Bim(\cl X)\subseteq \cl B(L^2(G))$
be the weak* closed masa-bimodule generated by $\cl X$.
The following theorem was proved in \cite{akt}.

\begin{theorem}\label{th_akt}
If $J\subseteq A(G)$ is a closed ideal then $\Sat(J)^{\perp} = \Bim(J^{\perp})$.
\end{theorem}

Given an ideal $J\subseteq A(G)$, we are interested in when
the subspace $J^{\perp}\cap C_r^*(G)$ is trivial.
For a subspace $\cl Y\subseteq C_r^*(G)$,
we therefore define
$$\Bim\mbox{}_0(\cl Y) = \overline{{\rm span}}\{M_b Y M_a :
Y\in \cl Y, \ a,b\in L^{\infty}(G) \mbox{ with compact support}\},$$
where the closure is taken in the norm topology. Note that, by the Stone-von Neumann Theorem,
$\Bim_0(C_r^*(G)) = \cl K$. In particular,  $\Bim_0(\cl Y)$ is a masa-bimodule consisting of compact operators,
for every subspace $\cl Y\subseteq C_r^*(G)$.

The following theorem can be viewed as a norm closure version of Theorem \ref{th_akt}.

\begin{theorem}\label{th_akt0}
If $J\subseteq A(G)$ is a closed ideal then
\begin{equation}\label{eq_bim0}
\Bim(J^{\perp})\cap \cl K = \Bim\mbox{}_0(J^{\perp}\cap C_r^*(G)).
\end{equation}
\end{theorem}
\begin{proof}
We first recall some technical tools from \cite{akt}, \cite{lt} and \cite{stt}.
Let $\widehat{G}$ be the set of all (equivalence classes of) irreducible representations of $G$.
For $\pi \in \widehat{G}$, acting on a Hilbert space $H_{\pi}$,
write $u^\pi_{i,j}$ for the function given by $u^\pi_{i,j}(r)=\sca{\pi(r)e_j,e_i}$,
where $\{e_n\}_{n\in \bb{N}_\pi}$ is a fixed orthonormal basis of $H_\pi$
(and the index set $\bb{N}_\pi$ has cardinality $\dim H_{\pi}$).
If $h\in T(G)$ is compactly supported and $\pi\in \widehat{G}$, define \cite{lt}
\begin{align}
h^\pi_{i,j}(s,t)& =\int_G h(sr,tr) u^\pi_{i,j}(r)dr; \nonumber \\
\tilde h^\pi_{i,j}(s,t)& =\int_G h(sr,tr) u^\pi_{i,j}(sr)dr .
\label{star}
\end{align}
If $\phi$ is a function on $G$, we denote by $\check\phi$
the function given by $\check\phi(s)=\phi(s^{-1})$, $s\in G$.

By \cite[Lemmas 3.4 and 3.8]{akt}, the functions of the form $\chi_{L\times L} h^{\pi}_{i,j}$ and $\chi_{L\times L} \tilde{h}^{\pi}_{i,j}$,
where $L\subseteq G$ is a compact subset and $h\in T(G)$ is compactly supported,
are Schur multipliers, and, by \cite[Lemma 3.12]{akt},
\begin{equation}\label{eq_series}
h^\pi_{i,j} \chi_{L\times L} =
\sum_k (\check u^\pi_{i,k}\otimes\mathbf{1})\tilde h^\pi_{k,j} \chi_{L\times L},
\end{equation}
where the convergence is in the norm of $T(G)$.

Let $\nph\in T(G)$. Then there exists \cite{stt} a (unique)
completely bounded map $E_{\nph} : \cl B(L^2(G))\to \vn(G)$ such that
$$\langle E_{\nph}(T),u\rangle = \langle T, \nph N(u)\rangle, \ \ \ \ u\in A(G).$$
Moreover, if $T\in \cl K$ then $E_{\nph}(T)\in C^*_r(G)$.

By the remarks preceding the formulation of Theorem \ref{th_akt0},
\begin{equation}\label{eq_inc0}
\Bim\mbox{}_0(J^{\perp}\cap C_r^*(G))\subseteq \Bim(J^{\perp})\cap \cl K.
\end{equation}
Let $T\in \Bim(J^{\perp})\cap \cl K$.
To show equality in (\ref{eq_inc0}), it suffices to prove that if
$h\in T(G)$ annihilates $\Bim_0(J^{\perp}\cap C_r^*(G))$ then $\langle T,h\rangle = 0$.
Suppose that $h\in \Bim_0(J^{\perp}\cap C_r^*(G))^{\perp}$.
By \cite[Lemma 3.13]{akt}, there are compact sets $K_n$, $n\in \bb{N}$, such that
$h\chi_{K_n\times K_n}\to h$ in the norm of $T(G)$. Moreover,
$h\chi_{K_n\times K_n} \in \Bim_0(J^{\perp}\cap C_r^*(G))^{\perp}$
since the latter space is a $\cl D$-bimodule.
If we show that $\langle T,h\chi_{K_n\times K_n}\rangle = 0$
then
$$\langle T,h\rangle = \lim_{n\to\infty} \langle T,h\chi_{K_n\times K_n}\rangle = 0;$$
we may thus assume that
$h = h\chi_{K\times K}$, for some compact set $K\subseteq G$.

Let $\nph\in T(G)$ and $u\in J$. Then $N(u)\in N(J)$ and hence
$\nph N(u)\in N(J) T(G)$. By Theorem \ref{th_akt}, $T\in \Sat(J)_{\perp}$ and hence
$$\langle E_{\nph}(T),u\rangle = \langle T, \nph N(u)\rangle = 0.$$
Thus, $E_{\nph}(T)\in J^{\perp}\cap C_r^*(G)$.
If $\nph\in T(G)$ and $a,b\in L^{\infty}(G)$ are compactly supported then, using (\ref{pair}), we have
\begin{eqnarray*}
0 & = & \langle M_a E_{\nph}(T) M_b, h \rangle
= \langle E_{\nph}(T), (a\otimes b)h \rangle
= \langle E_{\nph}(T), P((a\otimes b)h) \rangle\\
& = & \langle T, \nph NP((a\otimes b)h) \rangle.
\end{eqnarray*}
Using (\ref{eq_p}), we obtain
\begin{eqnarray*}
& &
\chi_{K\times K} NP((a\otimes b)h) (s,t)
 =
\chi_K(s)\chi_K(t) P((a\otimes b)h)(ts^{-1})\\
& = &
\chi_K(s)\chi_K(t) \int ((a\otimes b)h)(st^{-1}x,x) dx\\
& = &
\chi_K(s)\chi_K(t) \int ((a\otimes b)h)(sr,tr) dr\\
& = &
\chi_K(s)\chi_K(t) \int_{K^{-1}K} a(sr)b(tr) h(sr,tr) dr\\
& = &
\chi_K(s)\chi_K(t) \int_{K^{-1}K} \chi_K(sr)a(sr)\chi_K(tr)b(tr) h(sr,tr) dr\\
& = &
\chi_{K\times K} NP((a\chi_K\otimes b\chi_K)h) (s,t).
\end{eqnarray*}
It follows that
\begin{eqnarray*}
\langle T, \chi_{K\times K} \tilde{h}^{\pi}_{i,j}\rangle
& = &
\langle T, \chi_{K\times K} NP((u^{\pi}_{i,j}\otimes 1)h)\rangle\\
& = &
\langle T, \chi_{K\times K} NP((u^{\pi}_{i,j}\chi_{K} \otimes \chi_K)h)\rangle = 0,
\end{eqnarray*}
for all $\pi\in \widehat{G}$ and all $i,j$.
By (\ref{eq_series}), $\langle T, \chi_{K\times K} h^{\pi}_{i,j}\rangle = 0$,
and by \cite[Lemma 3.14]{akt}, $\langle T,h\rangle = 0$.
\end{proof}

As an immediate consequence of Theorem \ref{th_akt0} we
obtain the following transference result.

\begin{corollary}\label{c_trans}
Let $J\subseteq A(G)$ be a closed ideal. The following are equivalent:

(i) \ $J$ is an ideal of multiplicity;

(ii) $\Bim(J^{\perp})$ is a bimodule of multiplicity.
\end{corollary}
\begin{proof}
(ii)$\Rightarrow$(i) If $J^{\perp}\cap C_r^*(G) = \{0\}$ then $\Bim_0(J^{\perp}\cap C_r^*(G)) = \{0\}$
and, by Theorem \ref{th_akt0}, $\Bim(J^{\perp})\cap \cl K =\{0\}$, a contradiction.

(i)$\Rightarrow$(ii) If $T$ is a non-zero operator in $J^{\perp}\cap C_r^*(G)$ then
there exist a compact subset $K$ of $G$ such that $M_{\chi_K} T M_{\chi_K} \neq 0$.
Thus, $\Bim_0(J^{\perp}\cap C_r^*(G)) \neq \{0\}$ and, by Theorem \ref{th_akt0},
$\Bim(J^{\perp})\cap \cl K \neq \{0\}$.
\end{proof}

For a closed subset $E\subseteq G$, let
$$E^* = \{(s,t)\in G\times G : ts^{-1}\in E\}.$$
Corollary \ref{c_trans} has the following consequences, originally established in \cite{stt}.

\begin{corollary}\label{c_E}
Let $E\subseteq G$ be a closed set. Then

(i) \ $E$ is an $M$-set if and only if $E^*$ is an operator $M$-set;

(ii) $E$ is an $M_1$-set if and only if $E^*$ is an operator $M_1$-set.
\end{corollary}
\begin{proof}
By \cite[Theorem 5.3]{akt},
$\frak{M}_{\max}(E^*) = \Bim(J(E)^{\perp})$ and $\frak{M}_{\min}(E^*) = \Bim(I(E)^{\perp})$.
The claims now follow from Corollary \ref{c_trans}.
\end{proof}

\section{Preservation}\label{s_pres}

In this section, we show that the property of being an ideal of multiplicity
(resp. uniqueness) is preserved under some natural operations.

\subsection{Intersections}
Suppose that $E_1, E_2\subseteq G$ are $U_1$-sets (that is, that they are not $M_1$-sets).
It was shown in \cite{stt} that $E_1\cup E_2$ is an $U_1$-set.
Since, trivially, $I(E_1)\cap I(E_2) = I(E_1\cup E_2)$,
this result equivalently says that if $I(E_1)$ and $I(E_2)$ are ideals of uniqueness
then so is their intersection.
This is the motivation behind the present subsection, whose main result,
Corollary \ref{c_int}, will be obtained as a consequence of the more general
Theorem \ref{th_sumsink}. We start with a lemma which we believe may be interesting in its own right.

\begin{lemma}\label{l_weaksums}
Let $(X,\mu)$ be a standard $\sigma$-finite measure space and
$\cl S,\cl T\subseteq T(X)$ be $\frak{S}(X)$-invariant subspaces.
If $\cl S + \cl T$ is weak* dense in $T(X)$ then
$$\overline{\cl S}^{w^*} \cap \overline{\cl T}^{w^*} = \overline{\cl S \cap \cl T}^{w^*}.$$
\end{lemma}
\begin{proof}
The inclusion
$$\overline{\cl S \cap \cl T}^{w^*}\subseteq \overline{\cl S}^{w^*} \cap \overline{\cl T}^{w^*}$$
is trivial.
Let $h\in \overline{\cl S}^{w^*} \cap \overline{\cl T}^{w^*}$.
Suppose that $X_n\subseteq X$, $n\in \bb{N}$, are measurable subsets of finite measure,
such that $X_n\subseteq X_{n+1}$ for every $n$ and
$X = \cup_{n=1}^{\infty} X_n$.
Since $\chi_{X_n\times X_n}\in \frak{S}(X)$, by weak* continuity we have that
$$\chi_{X_n\times X_n} h\in \overline{\cl S}^{w^*} \cap \overline{\cl T}^{w^*}, \ \ \ n\in \bb{N}.$$
Since $\chi_{X_n\times X_n} h\stackrel{w^*}{\longrightarrow}_{n\to\infty} h$,
it suffices to show that $\chi_{X_n\times X_n} h\in \overline{\cl S \cap \cl T}^{w^*}$ for every $n\in \bb{N}$.
We may thus assume that the measure $\mu$ is finite.
Note that, in this case, $\frak{S}(X)\subseteq T(X)$, as the constant function $1$ is in $T(X)$.

Write $h = \sum_{i=1}^{\infty} f_i\otimes g_i$, where $(f_i)_{i\in \bb{N}}$ and $(g_i)_{i\in \bb{N}}$
are square summable sequences in $L^2(X)$.
For $N\in \bb{N}$, let
$$Y_N = \{x\in X : \sum_{i=1}^{\infty} |f_i(x)|^2 \leq N\}$$
and
$$Z_N = \{x\in X : \sum_{i=1}^{\infty} |g_i(x)|^2 \leq N\}.$$
Then $(Y_N)_{N\in \bb{N}}$ (resp. $(Z_N)_{N\in \bb{N}}$)
is an ascending sequence whose union has full measure in $X$.
Moreover, $\chi_{Y_n\times Z_N} h\in \frak{S}(X)$ , $N\in \bb{N}$ (see \cite{peller}),
and $\chi_{Y_n\times Z_N} h \to h$ in the norm of $T(X)$.
We thus showed that 
\begin{equation}\label{eq_fr}
\cl R\subseteq \overline{\frak{S}(X)\cap \cl R}^{\|\cdot\|},
\end{equation}
for every $\frak{S}(X)$-invariant subspace $\cl R$ of $T(X)$;
we may hence assume that $h\in \frak{S}(X)$.

By (\ref{eq_fr}), $\overline{\cl S}^{w^*} = \overline{\frak{S}(X)\cap \cl S}^{w^*}$.
Thus, there exists a net $(h_i)\subseteq \frak{S}(X)\cap  \cl S$
such that $h = $w$^*$-$\lim_i h_i$.
Since $\cl S$ and $\cl T$ are invariant under $\frak{S}(X)$,
it follows that if $\psi\in \frak{S}(X)\cap \cl T$
then
$$\psi h = \mbox{w}^*\mbox{-}\lim\mbox{}_i \psi h_i\in \overline{\cl S\cap \cl T}^{w^*}.$$
Similarly, if $\nph\in \frak{S}(X)\cap \cl S$ then $\nph h\in \overline{\cl S\cap \cl T}^{w^*}$.
Thus,
$$\theta h\in \overline{\cl S\cap \cl T}^{w^*}, \ \mbox{ for every }
\theta\in \frak{S}(X) \cap \cl S + \frak{S}(X) \cap \cl T.$$
Inclusion (\ref{eq_fr}) 
implies that $\frak{S}(X) \cap \cl S + \frak{S}(X) \cap \cl T$ is weak* dense in $T(X)$.
Thus, the constant function $\chi_{X\times X}$ is the limit
of a net $(\theta_i)_i\subseteq \frak{S}(X) \cap \cl S + \frak{S}(X) \cap \cl T$
in the weak* topology of $T(X)$.
It follows that $h = \chi_{X\times X}h$ is the limit of the net
$(\theta_i h)_i$ in the weak* topology of $T(X)$, and hence $h\in \overline{\cl S\cap \cl T}^{w^*}$.
\end{proof}

\begin{theorem}\label{th_sumsink}
Let $(X,\mu)$ be a standard $\sigma$-finite measure space and
$\cl U,\cl V\subseteq \cl B(L^2(X))$ be weak* closed masa-bimodules.
If $\cl U \cap \cl V$ is a bimodule of uniqueness then
$$\overline{\cl U + \cl V}^{w^*} \cap \cl K = \overline{\cl U\cap \cl K + \cl V \cap \cl K}^{\|\cdot\|}.$$
\end{theorem}
\begin{proof}
Let $\cl S = \cl U_{\perp}$ and $\cl T = \cl V_{\perp}$.
Since $\cl U$ and $\cl V$ are weak* closed masa-bimodules,
$\cl S$ and $\cl T$ are invariant under $\frak{S}(X)$.
Furthermore,
$$\overline{\cl S + \cl T}^{w^*} = (\cl U\cap \cl V\cap \cl K)^{\perp} = T(X).$$

It is trivial that
$$\overline{\cl U\cap \cl K + \cl V \cap \cl K}^{\|\cdot\|} \subseteq \overline{\cl U + \cl V}^{w^*} \cap \cl K.$$
To show equality, suppose that $h\in T(X)$ annihilates $\cl U\cap \cl K + \cl V \cap \cl K$.
Then $h\in \overline{\cl S}^{w^*}\cap \overline{\cl T}^{w^*}$ and hence, by Lemma \ref{l_weaksums},
$h\in  \overline{\cl S \cap \cl T}^{w^*}$. Since each element of $\cl S \cap \cl T$
annihilates $\overline{\cl U + \cl V}^{w^*}$, we have that $h$ annihilates
$\overline{\cl U + \cl V}^{w^*} \cap \cl K$.
The proof is complete.
\end{proof}

\begin{corollary}\label{c_sumsuniq}
Let $(X,\mu)$ be a standard $\sigma$-finite measure space and
$\cl U,\cl V\subseteq \cl B(L^2(X))$ be weak* closed masa-bimodules.
If $\cl U$ and $\cl V$ are bimodules of uniqueness then so is
$\overline{\cl U + \cl V}^{w^*}$.
\end{corollary}
\begin{proof}
Since $\cl U$ is a bimodule of uniqueness, so is $\cl U\cap \cl V$.
The conclusion now follows from Theorem \ref{th_sumsink}.
\end{proof}

For the rest of this subsection we assume that $G$ is a second countable locally compact group.

\begin{corollary}\label{c_sumsij}
Let $J_1$ and $J_2$ be
closed ideals of $A(G)$. If $\overline{J_1 + J_2}$ is an ideal of uniqueness then
$$\overline{\Bim\mbox{}_0(J_1^{\perp}\cap C_r^*(G)) + \Bim\mbox{}_0(J_2^{\perp}\cap C_r^*(G))}^{\|\cdot\|}
= \Bim\mbox{}_0((J_1 \cap J_2)^{\perp}\cap C_r^*(G)).$$
\end{corollary}
\begin{proof}
By \cite[Corollary 4.4]{akt},
\begin{equation}\label{eq_j1j2}
\overline{\Bim(J_1^{\perp}) + \Bim(J_2^{\perp})}^{w^*} = \Bim((J_1\cap J_2)^{\perp})
\end{equation}
and
\begin{equation}\label{eq_j1j22}
\Bim(J_1^{\perp}) \cap \Bim(J_2^{\perp}) = \Bim((J_1+ J_2)^{\perp}).
\end{equation}
By (\ref{eq_j1j22}), Theorem \ref{th_akt0} and the assumption that $\overline{J_1 + J_2}$ is an ideal of uniqueness,
we have that
$$\Bim(J_1^{\perp}) \cap \Bim(J_2^{\perp})\cap \cl K =  \Bim\mbox{}_0((J_1 + J_2)^{\perp}\cap C_r^*(G)) = \{0\},$$
that is, $\Bim(J_1^{\perp}) \cap \Bim(J_2^{\perp})$ is a bimodule of uniqueness.
By (\ref{eq_j1j2}) and Theorem \ref{th_sumsink},
$$\Bim((J_1\cap J_2)^{\perp})\cap \cl K =
\overline{\Bim(J_1^{\perp})\cap \cl K + \Bim(J_2^{\perp})\cap \cl K}^{\|\cdot\|}.$$
The claim is now immediate from Theorem \ref{th_akt0}.
\end{proof}

\begin{corollary}\label{c_int}
Let $J_1$ and $J_2$ be ideals of uniqueness of $A(G)$.
Then $J_1\cap J_2$ is an ideal of uniqueness.
\end{corollary}
\begin{proof}
Since $J_1$ is an ideal of uniqueness, so is $\overline{J_1 + J_2}$.
The claim now follows from Corollary \ref{c_sumsij}.
\end{proof}

It is easy to see that the property of being an ideal of uniqueness is not preserved
under countable intersections. For an example, let $G = \bb{R}$ and $\{r_k\}_{k\in \bb{N}}$
be an enumeration of the rationals. Set $J_k = I(\{r_k\})$, $k\in \bb{N}$.
Then $J_k^{\perp} = \bb{C}\lambda_{r_k}$ and hence $J_k$ is an ideal of uniqueness;
however, $\cap_{k=1}^{\infty} J_k = \{0\}$ is not.

\subsection{Tensor products}
Let $G_1$ and $G_2$ be second countable locally compact groups.
Suppose that $E_i\subseteq G_i$, $i = 1,2$, are $M$-sets (resp. $M_1$-sets).
It was shown in \cite[Corollary 4.14]{stt} that, in this case,
$E_1\times E_2$ is an $M$-set (resp. $M_1$-set).
We now generalise this fact to closed ideals.
We denote by $\otimes$ the algebraic tensor product of vector spaces.
Note that $A(G_1)\otimes A(G_2)$ can be considered as a
dense subalgebra of the Fourier algebra $A(G_1\times G_2)$;
in fact, the latter can be canonically identified with the
operator projective tensor product $A(G_1)\otimes_{\wedge} A(G_2)$ of $A(G_1)$ and $A(G_2)$
(we refer the reader to \cite{blm} for background on this tensor product).

\begin{theorem}\label{th_tp}
Let
$J_1\subseteq A(G_1)$ and $J_2\subseteq A(G_2)$ be closed ideals.
Set
$$J = \overline{ J_1\otimes A(G_2) + A(G_1)\otimes J_2}^{\|\cdot\|_A}.$$
The following are equivalent:

(i) \ $J_1$ and $J_2$ are ideals of multiplicity;

(ii) $J$ is an ideal of multiplicity.
\end{theorem}
\begin{proof}
For $v\in A(G_2)$, let 
$$L_v : \vn (G_1)\bar\otimes\vn (G_2)\to \vn (G_1)$$
be the Tomiyama left slice map associated with $v$,
defined by the identity
$$\langle L_v(T),u\rangle = \langle T, u\otimes v\rangle, \ \ \ u\in A(G_1)$$
(see {\it e.g.} \cite{kraus_tams}).
For $u\in A(G_1)$, let
$$R_u : \vn (G_1)\bar\otimes\vn (G_2)\to \vn (G_2)$$
be the analogously defined right slice map.

Set $\cl X_i = J_i^{\perp}$, $i = 1,2$.
Let $\cl X_1\bar\otimes\cl X_2$ be the weak* closed spacial tensor product,
and
$$\cl X_1\bar\otimes_{\cl F} \cl X_2 =
\{T\in \cl B(L^2(G_1\times G_2)) : R_u(T)\in \cl X_2, L_v(T)\in \cl X_1,$$
$$\mbox{ for all } u\in A(G_1), v\in A(G_2)\}$$
be the normal Fubini tensor product (see \cite{kraus_tams}), of $\cl X_1$ and $\cl X_2$.
We have
\begin{equation}\label{eq_x12}
\cl X_1\bar\otimes\cl X_2 \subseteq (\cl X_1\bar\otimes \vn(G_2))\cap (\vn(G_1)\bar\otimes\cl X_2) = J^{\perp}
\subseteq \cl X_1\bar\otimes_{\cl F} \cl X_2.
\end{equation}

(i)$\Rightarrow$(ii)
Suppose that $T_i\in \cl X_i\cap C_r^*(G_i)$ is non-zero, $i = 1,2$. Then,
after identifying the minimal tensor product $C_r^*(G_1)\otimes_{\min} C_r^*(G_2)$ with
$C_r^*(G_1\times G_2)$ in the canonical way, we have that
$T_1\otimes T_2$ is a non-zero operator in
$(\cl X_1\bar\otimes\cl X_2)\cap C_r^*(G_1\times G_2)$.
Inclusion (\ref{eq_x12}) shows that $J$ is an ideal of multiplicity.

(ii)$\Rightarrow$(i) It suffices, by symmetry, to show that $J_1$ is an ideal of multiplicity.
Suppose, by way of contradiction, that $J_1$ is an ideal of uniqueness.
Recall that every $v\in A(G_2)$ can be viewed as a weak* continuous functional on $\vn(G_2)$
and, since $A(G_2)$ is a subspace of the reduced Fourier-Stieltjes algebra $B_r(G_2)$ (see \cite{eymard}),
as a norm continuous functional on $C_r^*(G_2)$.
Since $L_v$ is norm continuous, it maps $C_r^*(G_1\times G_2)$ into $C_r^*(G_1)$.
Suppose that $T$ is a non-zero operator in $J^{\perp}\cap C_r^*(G_1\times G_2)$.
It follows from (\ref{eq_x12}) that $L_v(T)\in \cl X_1\cap C_r^*(G_1)$.
Since $J_1$ is an ideal of uniqueness, $L_v(T) = 0$.
Thus, for every $u\in A(G_1)$ and every $v\in A(G_2)$, we have
$$\langle T,u\otimes v\rangle = \langle L_v(T),u\rangle = 0$$
and since $A(G_1)\otimes A(G_2)$ is dense in $A(G_1\times G_2)$,
we conclude that $T = 0$, a contradiction.
\end{proof}

\begin{corollary}\label{c_ampl}
 A closed ideal
$J_1$ of $A(G_1)$ is an ideal of multiplicity if and only if the ideal
$\overline{J_1\otimes A(G_2)}^{\|\cdot\|_A}$ of $A(G_1\times G_2)$ is an ideal of multiplicity.
\end{corollary}
\begin{proof}
Immediate from Theorem \ref{th_tp} after letting $J_2 = \{0\}$.
\end{proof}

\begin{theorem}\label{th_tp2}
Let
$J_1\subseteq A(G_1)$ and $J_2\subseteq A(G_2)$ be closed ideals.
Set
$$J = \overline{J_1\otimes A(G_2)}^{\|\cdot\|_A}\cap \overline{A(G_1)\otimes J_2}^{\|\cdot\|_A}.$$
The following are equivalent:

(i) \ $J_1$ and $J_2$ are ideals of uniqueness;

(ii) $J$ is an ideal of uniqueness.
\end{theorem}
\begin{proof}

(i)$\Rightarrow$(ii)
By Corollary \ref{c_ampl},
$\overline{J_1\otimes A(G_2)}^{\|\cdot\|_A}$ and
$\overline{A(G_1)\otimes J_2}^{\|\cdot\|_A}$ are ideals of uniqueness, and by Corollary \ref{c_int},
so is $J$.

(ii)$\Rightarrow$(i)
Set $\cl X_i = J_i^{\perp}$, $i = 1,2$. Note that
$$\overline{\vn(G_1)\otimes \cl X_2 + \cl X_1\otimes \vn(G_2)}^{w^*}\subseteq J^{\perp}.$$
If $T_1\in \cl X_1\cap C_r^*(G_1)$ and $T_2\in C_r^*(G_2)$ are non-zero then
$$T_1\otimes T_2 \in (\cl X_1\bar\otimes \vn(G_2))\cap C_r^*(G_1\times G_2)$$
is non-zero,
a contradiction with the assumption that $J$ is an ideal of uniqueness.
It follows that $J_1$ is an ideal of uniqueness, and by symmetry, so is $J_2$.
\end{proof}

The Remark after Definition \ref{d_mi}, Theorem \ref{th_tp2} and the fact that
\begin{equation}\label{eqq_j1j2}
\overline{J_1\otimes J_2}^{\|\cdot\|_A} \subseteq
\overline{J_1\otimes A(G_2)}^{\|\cdot\|_A}\cap \overline{A(G_1)\otimes J_2}^{\|\cdot\|_A}
\end{equation}
show that if
$\overline{J_1 \otimes J_2}^{\|\cdot\|_A}$ is an ideal of uniqueness then so are $J_1$ and $J_2$.
We do not know if the converse is true. In the next proposition, however, we give some
sufficient conditions which imply an equality in (\ref{eqq_j1j2}).

\begin{proposition} \label{bai}
Let  $J_1\subseteq A(G_1)$, $J_2\subseteq A(G_2)$ be closed ideals.
Assume that $G_1$ is an amenable locally compact group and $J_2$ has bounded approximate identity.
Then
$$\overline{J_1\otimes A(G_2)}^{\|\cdot\|_A} \cap \overline{A(G_1)\otimes J_2}^{\|\cdot\|_A}
= \overline{J_1\otimes J_2}^{\|\cdot\|_A}. $$
\end{proposition}
\begin{proof}
Let $(a_\alpha)_{\alpha}$ be a bounded approximate identity for $A(G_1)$,
which exists due to the amenability of $G_1$, and let $(b_\alpha)_{\alpha}$ be a
bounded approximate identity for $J_2$.
Let $C > 0$ be such that $\|a_{\alpha}\|\leq C$ and  $\|b_{\alpha}\|\leq C$ for all $\alpha$.

Fix $u\in \overline{J_1\otimes A(G_2)}^{\|\cdot\|_A}\cap\overline{A(G_1)\otimes J_2}^{\|\cdot\|_A}$, and
let $u_n\in J_1\otimes A(G_2)$ and $v_n\in A(G_1)\otimes J_2$ be such that $u=\lim_n u_n=\lim_n v_n$.
For every $\alpha$, we have that
$u(a_\alpha\otimes b_\alpha)=\lim_n u_n(a_\alpha\otimes b_\alpha)$ and since
$u_n(a_\alpha\otimes b_\alpha)\in J_1\otimes J_2$, for every $n$ and $\alpha$,
it follows that $u(a_\alpha\otimes b_\alpha)\in\overline{J_1\otimes J_2}$ for every $\alpha$.

As $v_n=\lim_\alpha v_n(a_\alpha\otimes b_\alpha)$ for any $n$, given $\varepsilon>0$, $n\in \mathbb N$,
there exists $\alpha(n)>0$ such that $\|v_n-v_n(a_{\alpha(n)}\otimes b_{\alpha(n)})\|<\varepsilon$.
Let $N$ be such that $\|u-v_n\|<\varepsilon$ whenever $n\geq N$.
If $n\geq N$ then
\begin{eqnarray*}
\|u - u(a_{\alpha(n)}\otimes b_{\alpha(n)})\|&\leq &\|u-v_n\|+\|v_n-v_n(a_{\alpha(n)}\otimes b_{\alpha(n)})\|\\
&+&\|v_n(a_{\alpha(n)}\otimes b_{\alpha(n)})-u(a_{\alpha(n)}\otimes b_{\alpha(n)})\|\\
&\leq& \varepsilon+\varepsilon+\varepsilon\|a_{\alpha(n)}\|\|b_{\alpha(n)}\|\leq (C^2 + 2)\varepsilon.
\end{eqnarray*}
Hence $u \in \overline{J_1\otimes J_2}$.
\end{proof}

Let $\cl R_c(G)$ be the collection of all subsets $E$ of $G$ of the form
\begin{equation}\label{form}
E=\cup_{i=1}^n(a_iH_i\setminus\cup_{j=1}^{m_i}b_{i,j}K_{i,j}),
\end{equation}
where $a_{i}$, $b_{i,j}\in G$, $H_i$ is a closed subgroup of $G$ and $K_{i,j}$ is an open subgroup of $H_i$ ($n,m_i\in\mathbb N_0$, $1\leq i\leq n$, $1\leq j\leq m_i$).
The family $\cl R_c(G)$ is known as the closed coset ring of $G$ (see \cite{forrest}).
By \cite{forest_kanuith_spronk},
if $G$ is an amenable locally compact group and $I$ is a closed ideal of $A(G)$, then $I$ has a bounded approximate identity
if and only if $I=I(E)$ for a subset $E\in\cl R_c(G)$. Note that, by \cite{forest_kanuith_spronk},
any element $E$ of $\cl R_c(G)$ is a set of spectral synthesis; hence $I(E)$ is an ideal of uniqueness
if and only if $E$ is a set of uniqueness.

\begin{corollary}\label{c_chmh}
Let $G_1$, $G_2$ be amenable second countable locally compact groups, $J_1\subseteq A(G_1)$
be a closed ideal  and $E$ be a subset of the form (\ref{form}).
The following are equivalent:

(i) \ $\overline{J_1\otimes I(E)}^{\|\cdot\|_A}$ is an ideal of uniqueness;

(ii) $J_1$ is an ideal of uniqueness and  $m_{G_2}(H_i)=0$, $1\leq i\leq n$.
\end{corollary}
\begin{proof}
By Theorem \ref{th_tp2}, Proposition \ref{bai} and the result stated before
the formulation of Corollary \ref{c_chmh},
it suffices to show that $I(E)$ is an ideal of uniqueness if and only if
$m_{G_2}(H_i) = 0$ for all $i=1,\ldots, n$.
Assume that $m_{G_2}(H_i)>0$ for some $i$, or, equivalently, that $H_i$ is open.
If $m_i=0$, then since the set
$a_iH_i$ is open, $m_{G_2}(a_iH_i)>0$ and hence $a_iH_i$ and $E$  are not  sets of uniqueness
(see \cite[Remark 4.3]{stt}).
 Let $m_i>0$. We have
 $$a_iH_i\setminus(\cup_{j=1}^{m_i} b_{i,j}K_{i,j})=a_i(H_i\setminus(\cup_{j=1}^{m_i} a_i^{-1}b_{i,j}K_{i,j})).$$
 As $K_{i,j}$ is open in $H_i$, the set $H_i\setminus (\cup_{j=1}^{m_i} a_i^{-1}b_{i,j}K_{i,j})$ is open and closed in $H_i$. Hence
 $m_{H_i}(H_i\setminus(\cup_{j=1}^{m_i} a_i^{-1}b_{i,j}K_{i,j}))>0$. As $m_{G_2}(H_i)>0$, by uniqueness of the Haar measure up to a constant, we have $m_{G_2}(\alpha)=Cm_{H_i}(\alpha)$ for some constant $C>0$ and any measurable $\alpha\subseteq H_i$. Hence $m_{G_2}(a_iH_i\setminus(\cup_{j=1}^{m_i} b_{i,j}K_{i,j}))>0$, giving that $a_iH_i\setminus(\cup_{j=1}^{m_i} b_{i,j}K_{i,j})$ and $E$ are not sets of uniqueness.

 Assume now that $m_{G_2}(H_i)=0$ for any $i=1,\ldots,n$. By \cite[Corollary 5.10 ]{stt}, $H_i$ is a set of uniqueness.
 Hence any subset of $a_iH_i$ is a set of uniqueness.
 Therefore $E$ is a set of uniqueness as a finite union of sets of uniqueness (\cite[Proposition 5.3]{stt_IEOT}).
\end{proof}

\begin{remark}\rm
It has been noted in \cite{ulger} that, for a connected amenable group $G$, any set $E$ in $\cl R_c(G)$ is a set of uniqueness.
This also follows from Corollary \ref{c_chmh}, as in this case  $G$ does not contain any proper open subgroup.
\end{remark}


\subsection{Inverse images}

In this subsection we establish an inverse image result for ideals of uniqueness.

\begin{theorem}\label{th_invima}
Let $G$ and $H$ be locally compact second countable groups with Haar measures $m_G$ and $m_H$, respectively.
Let $\nph:G\to H$  be a continuous homomorphism.
Assume that the measure $\nph_*m_G$ is absolutely continuous with respect to $m_H$.
Let $J$ be a closed ideal of $A(H)$ and
$\nph_*(J)$ be the closed ideal of $A(G)$ generated by the set $\{f\circ \nph : f\in J\}$.
The following hold:

(i) If $\nph$ is injective and has a continuous inverse on $\nph(G)$, then $\nph_*(J)$ is an ideal of uniqueness whenever  $J$ is so.

(ii) If $\nph_*m_G$ is equivalent to $m_H$ then
$\nph_*(J)$ is an ideal of multiplicity whenever $J$ is so.
\end{theorem}
\begin{proof}
Let $r : H\rightarrow \bb{R}^+$ be the Radon-Nikodym derivative of $\nph_*m_G$
with respect to $m_H$; thus,
$$m_G(\nph^{-1}(\alpha)) = \int_{\alpha} r(x)dm_H(x)$$
for every measurable subset $\alpha\subseteq H$.
Letting $M = \{x\in H : r(x) = 0\}$,
note that
$$m_G(\nph^{-1}(M)) = \int_{M} r(x)dm_H(x) = 0.$$

Let  $V_{\nph} : L^2(H)\rightarrow L^2(G)$ be given by
$$V_{\nph}\xi (x) =
\begin{cases}
\frac{\xi(\nph(x))}{\sqrt{r(\nph(x))}} \;&\text{if }x\not\in \nph^{-1}(M),\\
0 &\text{if }x\in \nph^{-1}(M).
\end{cases}$$
It was shown in \cite[Lemma 5.4]{stt} that $V_\varphi$  is a partial isometry with initial space $L^2(M^c,m_H|_{M^c})$.
Moreover, if $\nph$ is injective then $V_{\nph}$ is surjective.

(i) We have
$$\Sat(\nph_*(J)) = \overline{[N(\nph_*(J))T(G)]}^{\|\cdot\|_T};$$
thus,
$\Sat(\nph_*(J))$ is the closed linear span of the functions of the form
$$(s,t)\to u(\nph(t)\nph(s)^{-1})h(s,t),$$
where $u \in J$ and $h\in T(G)$.
Let $\Theta$ be the linear map on the algebraic tensor product
$L^2(H)\otimes L^2(H)$ given by
$\Theta (f\otimes g)= V_\nph f\otimes V_\nph g$. As $V_\nph$ is a partial isometry, $\Theta$ can be extended to a contractive map $\Theta:T(H)\to T(G)$.
It was shown in the proof of \cite[Theorem 5.5]{stt} that if $h\in T(G)$ then
\begin{equation}\label{eq_eft}
\Theta(h)(x,y) = \frac{h(\nph(x),\nph(y))}{\sqrt{r(\nph(x)) r(\nph(y))}}, \
\mbox{ for m.a.e. } (x,y)\in \nph^{-1}(M)\times \nph^{-1}(M),
\end{equation}
and that the map $\Theta$ is the adjoint of the map on
$\cl K(L^2(G))$ sending an element $T\in \cl K(L^2(G))$ to the operator $V_\nph^*KV_\nph$.
It follows that $\Theta$ is weak* continuous and hence, if $\cl M\subseteq T(G)$ then
$\Theta(\overline{\cl M}^{w^*})\subseteq\overline{\Theta(\cl M)}^{w^*}$.

Suppose that $u\in J$, $\psi\in T(G)$ and $h = N(u)\psi$. Then
$$\Theta(h)(s,t) =
\frac{N(f\circ \nph)(s,t)\psi(\nph(s),\nph(t))}{\sqrt{r(\nph(s)) r(\nph(t))}}, \ \ \ s,t\in G,$$
and so $\Theta(h)\in \Sat(\nph_*(J))$.

Suppose that $J$ is an ideal of uniqueness.
By Corollary \ref{c_trans} and Theorem \ref{th_akt}, $\overline{\Sat(J)}^{w^*}=T(H)$.
As $\nph$ is injective, $V_\nph$ is surjective and hence the image of $\Theta$ is dense in $T(G)$.
We now have
$$T(G)=\overline{\Theta(T(H))}^{\|\cdot\|_T}
=\overline{\Theta(\overline{\Sat(J)}^{w^*})}^{\|\cdot\|_T}
\subseteq \overline{\Theta(\overline{\Sat(J)}^{w^*})}^{w^*}\subseteq\overline{\Sat(\nph_*(J))}^{w^*};$$
thus, $\overline{\Sat(\nph_*(J))}^{w^*} = T(H)$, and by Corollary \ref{c_trans} and Theorem \ref{th_akt},
$\nph_*(J)$ is an ideal of uniqueness.

(ii)
Suppose that $J$ is an ideal of multiplicity and let
$T_1$ be a non-zero compact operator in
$\Sat(J)^\perp$. Let $T = V_\nph T_1V_\nph^*$. As $\nph_*m_G$ is equivalent to $m_H$,
the set $M$ is $m_H$-null, and hence
$V_\nph^*V_\nph=I$. It follows that $T_1 = V_\nph^*TV_\nph$ and hence $T$ is a non-zero compact operator.

We claim that, given $h\in T(G)$ and $\phi\in \frak{S}(G)$,
\begin{equation}\label{eq_phipsi}
(V_\nph^*\otimes V_\nph^*)((\phi\circ(\nph\otimes\nph))h) = \phi (V_\nph^*\otimes V_\nph^*)(h).
\end{equation}
To see (\ref{eq_phipsi}), let first $a\in L^\infty(H)$ and $\xi\in L^2(G)$.
Then, for every $\eta\in L^2(H)$, we have
\begin{eqnarray*}
\langle V_\nph^*(a\circ\nph)\xi, \eta\rangle&=&\langle (a\circ\nph)\xi, V_\nph \eta\rangle
= \int a(\nph(x))\frac{\overline{\eta(\nph(x))}}{\sqrt{r(\nph(x))}}\xi(x)dm_G(x)\\
&=&\langle \xi, V_\nph\bar a \eta\rangle=
\langle M_a V_\nph^*\xi,\eta\rangle.
\end{eqnarray*}
Thus, $V_\nph^*(a\circ\nph)\xi =M_a V_\nph^*\xi$, and it follows that
(\ref{eq_phipsi}) holds whenever $\phi$ and $h$ are elementary tensors.
By linearity, it holds whenever $\phi\in L^{\infty}(G)\otimes L^{\infty}(G)$ and
$h \in L^2(G)\otimes L^2(G)$.

Now suppose that $h = \sum_{i=1}^{\infty} f_i\otimes g_i$, where
$(f_i)_{i\in \bb{N}}$ and $(g_i)_{i\in \bb{N}}$ are square summable sequences in $L^2(G)$.
Setting $h_k = \sum_{i=1}^{k} f_i\otimes g_i$, $k\in \bb{N}$, we have that
$\|h - h_k\|_T\to 0$ and the continuity of Schur multiplication and that of the operator $V_\nph^*\otimes V_\nph^*$
imply that (\ref{eq_phipsi}) holds for every $h\in T(G)$ and every $\phi\in L^{\infty}(G)\otimes L^{\infty}(G)$.
Finally, if $\phi$ is arbitrary, it is the limit in the Schur multiplier norm of a sequence $(\phi_k)_{k\in \bb{N}}$
in $L^{\infty}(G)\otimes L^{\infty}(G)$.
Since $\phi_k\circ (\nph\otimes\nph)\to \phi\circ (\nph\otimes\nph)$ in the Schur multiplier norm,
(\ref{eq_phipsi}) is established in full generality.

Let now $u\in J$, $\psi\in T(G)$. Then
\begin{eqnarray*}
\langle T, (N(u)\circ(\nph\otimes\nph))\psi\rangle &=&\langle T_1, (V_\nph^*\otimes V_\nph^*)(N(u)\circ (\nph\otimes\nph))\psi)\rangle\\
&=&\langle T_1, N(u)(V_\nph^*\otimes V_\nph^*)(\psi)\rangle=0.
\end{eqnarray*}
Hence $T\in \Sat(\nph_*(J))^\perp$, and by Corollary \ref{c_trans} and Theorem \ref{th_akt}, $\nph_*(J)$ is an  ideal of multiplicity.
\end{proof}

\end{document}